\documentclass{amsart}

\usepackage{amsmath}
\usepackage{amsthm}
\usepackage{amsfonts}
\usepackage{amssymb}
\usepackage{mathabx}

\usepackage{graphics}
\usepackage{epsfig}
\usepackage{color}
\usepackage{verbatim}

\allowdisplaybreaks

\usepackage[utf8]{inputenc}
\usepackage{enumerate}
\usepackage[bookmarks = true, colorlinks=true, linkcolor = black, citecolor = black, menucolor = black, urlcolor = black]{hyperref}

\newcommand\R{{\mathbb{R}}}
\newcommand\T{{\mathbb{T}}}
\newcommand\N{{\mathbb{N}}}

\newcommand\supp{\operatorname{supp}}

\newcommand\CA{{\mathcal A}}

\newcommand\CC{{\mathcal C}}
\newcommand\CD{{\mathcal D}}

\newcommand\CS{{\mathcal S}}

\DeclareMathOperator{\pv}{p.v.}

\theoremstyle{plain}
  \newtheorem{theorem}{Theorem}

  \newtheorem{proposition}[theorem]{Proposition}

  \newtheorem{corollary}[theorem]{Corollary}

\theoremstyle{definition}

  \newtheorem{remark}[theorem]{Remark}

\numberwithin{theorem}{section}

\date{\today}
\author{David Beltran}
\address{School of Mathematics\\
University of Birmingham\\
Watson Building\\
Edgbaston\\
Birmingham\\
B15 2TT\\
United Kingdom}
\email{dbeltran89@gmail.com}
\thanks{This work was supported by the European Research Council [grant
number 307617]}
\title{A Fefferman--Stein inequality for the Carleson operator}

\begin{document}

%
%

\begin{abstract}
We provide a Fefferman--Stein type weighted inequality for maximally modulated Calderón--Zygmund operators that satisfy \textit{a priori} weak type unweighted estimates. This inequality corresponds to a maximally modulated version of a result of Pérez. Applying it to the Hilbert transform we obtain the corresponding inequality for the Carleson operator $\mathcal{C}$, that is $\mathcal{C}: L^p(M^{\lfloor p \rfloor +1}w) \to L^p(w)$ for any $1<p<\infty$ and any weight function $w$, with bound independent of $w$. We also provide a maximal-multiplier weighted theorem, a vector-valued extension, and more general two-weighted inequalities. Our proof builds on a recent work of Di Plinio and Lerner combined with some results on Orlicz spaces developed by Pérez.
\end{abstract}

%
%

\maketitle

\section{Introduction}
Let $M$ denote the Hardy--Littlewood maximal operator. In 1971, Fefferman and Stein \cite{FS} proved that there is a constant $C_n<\infty$ \footnote{Here and throughout the paper we use the letter $C$ to denote a constant that may change from line to line that is, in particular, independent of the function $f$ and the weight $w$. We also use the notation $A \lesssim B$ to denote that there is a constant $C$ such that $A \leq CB$.} such that for any $1<p<\infty$ 
\begin{equation}\label{FSineq}
\int_{\R^n}|Mf(x)|^p w(x)dx \leq C_n p' \int_{\R^n}|f(x)|^p Mw(x)dx
\end{equation}
holds for all weight functions $w$, where $p'$ denotes the conjugate exponent of $p$, that is $1/p+1/p'=1$. By weight we mean a non-negative locally integrable function.

The question of the existence of inequalities like \eqref{FSineq} has been raised on a more general level. More precisely, given an operator $U$ and an exponent $p \in [1,\infty)$, one may attempt to identify a controlling maximal operator $\mathfrak{M}$ for which there is a constant $C<\infty$ such that
\begin{equation}\label{GeneralFS}
\int_{\R^n} |Uf(x)|^pw(x)dx \leq C \int_{\R^n} |f(x)|^p \mathfrak{M}w(x)dx
\end{equation}
holds for all admissible functions $f$ and weights $w$. An inequality like \eqref{GeneralFS}, together with an elementary duality argument, allows to transfer bounds from $\mathfrak{M}$ to $U$, that is,
\begin{equation}\label{Bounds}
\|U\|_{L^q \to L^{\tilde{q}}} \lesssim \|\mathfrak{M}\|_{L^{(\tilde{q}/p)'} \to L^{(q/p)'}}^{1/p} 
\end{equation}
for $q,\tilde{q}\geq p$. This induces a concept of \textit{optimality} in the maximal operator $\mathfrak{M}$ and the weighted inequalities \eqref{GeneralFS}; for any fixed exponent $p$, we would like to determine a maximal operator $\mathfrak{M}$ whose Lebesgue space bounds allow to recover optimal Lebesgue space bounds for $U$ via \eqref{Bounds}.

In this paper we address the problem of controlling the Carleson operator and more general maximally modulated Calderón--Zygmund operators by an optimal maximal function. Despite the oscillatory nature of these operators, the techniques we make use of are closer to Calderón--Zygmund theory. We start therefore with a careful inspection of inequalities of the type \eqref{GeneralFS} in the context of Calderón--Zygmund operators.

A Calderón--Zygmund operator $T$ on $\R^n$ is an $L^2$-bounded operator represented as
$$
Tf(x)=\int_{\R^n}K(x,y)f(y)dy, \:\:\:\:\: x \not \in \supp f,
$$
where the kernel $K$ satisfies
\begin{enumerate}[(i)]
\item $|K(x,y)|\leq \frac{C}{|x-y|^n}$ for all $x \neq y$;
\item $|K(x,y)-K(x',y)|+|K(y,x)-K(y,x')|\leq C\frac{|x-x'|^\delta}{|x-y|^{n+\delta}}$ for some $0<\delta\leq 1$ when $|x-x'|<|x-y|/2$.
\end{enumerate}
Córdoba and Fefferman \cite{CF} showed that for $s>1$ and $1<p<\infty$ there is a constant $C<\infty$ such that
\begin{equation}\label{CFineq}
\int_{\R^n}|Tf(x)|^pw(x)dx \leq C \int_{\R^n} |f(x)|^p M_sw(x)dx
\end{equation}
holds for any weight $w$, where $M_sw(x):=(Mw^s(x))^{1/s}$.\footnote{This can also be seen as a consequence of the $A_p$ theory, since $M_sw \in A_1 \subset A_p$ for $p>1$, with constant independent of $w$, and $w\leq M_sw$.} Observe that the operator $M_s$ is not an \textit{optimal} maximal operator in the sense described by \eqref{Bounds}. For each $s>1$, $M_s$ fails to be bounded on $L^r$ for $1<r\leq s$. Thus, for a fixed $1<p<\infty$, the mechanism \eqref{Bounds} only allows one to recover $L^q$ bounds for $T$ in the restricted range $p\leq q<ps'$, missing the exponents in $[ps',\infty)$; we recall that $T$ is an $L^q$ bounded operator for $1<q<\infty$. This problem was resolved by Wilson \cite{Wi89} in the range $1<p\leq 2$ and by Pérez \cite{Pe94} in the whole range $1<p<\infty$, who showed that for $1<p<\infty$, there is a constant $C<\infty$ such that
\begin{equation}\label{WPineq}
\int_{\R^n}|Tf(x)|^p w(x)dx \leq C \int_{\R^n}|f(x)|^p M^{\lfloor p \rfloor +1}w(x)dx
\end{equation}
holds for any weight $w$. Here $\lfloor p \rfloor$ denotes the integer part of $p$ and $M^{\lfloor p \rfloor +1}$ denotes the ($\lfloor p \rfloor +1$)-fold composition of $M$. The operator $M^{\lfloor p \rfloor +1}$ is bounded on $L^r$, $1<r<\infty$, for any $p$. Thus, given $1<p<\infty$, it is an optimal maximal function as we can recover via \eqref{Bounds} the $L^q$ boundedness of $T$ for the whole range $p \leq q<\infty$. Furthermore, their result is best possible in the sense that \eqref{WPineq} fails if $M^{\lfloor p \rfloor +1}$ is replaced by $M^{\lfloor p \rfloor}$. It should be noted that for each $s>1$ and $k\geq 1$, the pointwise estimate $M^kw(x) \leq C M_sw(x)$ holds for some constant $C$ independent of $w$.

Our goal is to extend \eqref{WPineq} to a broad class of maximally modulated Calderón--Zygmund operators studied previously by Grafakos, Martell and Soria \cite{GMS} and Di Plinio and Lerner \cite{DiPL}. Let $\Phi=\{\phi_\alpha\}_{\alpha \in A}$ be a family of real-valued measurable functions indexed by an arbitrary set $A$. The \textit{maximally modulated Calderón--Zygmund} operator $T^{\Phi}$ is defined by
\begin{equation}\label{modulated}
T^{\Phi}f(x):=\sup_{\alpha \in A} |T(\mathcal{M}^{\phi_{\alpha}}f)(x)|,
\end{equation}
where $\mathcal{M}^{\phi_\alpha}f(x):=e^{2\pi i \phi_\alpha(x)}f(x)$. We will consider operators $T^{\Phi}$ such that for some $r_0>1$ satisfy the \textit{a priori} weak-type unweighted inequalities
\begin{equation}\label{wt}
\|T^{\Phi}f\|_{r,\infty} \lesssim \psi(r)\|f\|_r
\end{equation}
for $1<r\leq r_0$, where $\psi(r)$ is a function that captures the dependence of the operator norm on $r$. This definition is of course motivated by the Carleson operator
\begin{equation*}
\mathcal{C}f(x)=\sup_{\alpha \in \R} \left|\pv \int_\R \frac{e^{2\pi i \alpha y }}{x-y}f(y)dy \right|,
\end{equation*}
since it may be recovered from \eqref{modulated} by setting $T=H$ and $\Phi$ to be the family of functions given by $\phi_\alpha(x)=\alpha x$ for $\alpha \in \R$. Expressing $\mathcal{C}f$ in terms of $\widehat{f}$ allows it to be reconciled with the classical expression for the Carleson maximal operator in terms of partial Fourier integrals. It is well known that the Carleson operator satisfies condition \eqref{wt} for $1<r<\infty$, implying almost everywhere convergence of Fourier series for functions in $L^r$; see Carleson \cite{Carleson} and Hunt \cite{Hunt} for the celebrated Carleson--Hunt theorem or Fefferman \cite{FeffermanCarleson} and Lacey and Thiele \cite{LT} for alternative proofs.

Implicit in the work of Di Plinio and Lerner \cite{DiPL} there is the following analogue of the estimate \eqref{CFineq} for maximally modulated Calderón--Zygmund operators.\footnote{Again, this result can be seen as a consequence of the $A_\infty$ theory in \cite{GMS}. In the case of the Carleson operator $\CC$ the result follows from the $A_p$ theory in \cite{HY}.}

\begin{theorem}\label{DLtheorem}
Let $T^{\Phi}$ be a maximally modulated Calderón--Zygmund operator satisfying \eqref{wt}. Then for $1<s<2$ and $1<p<\infty$ there is a constant $C<\infty$ such that for any weight $w$
\begin{equation}\label{FSDL}
\int_{\R^n} |T^{\Phi}f(x)|^p w(x)dx \leq C \int_{\R^n} |f(x)|^p M_sw(x)dx.
\end{equation}
\end{theorem}

Note that the inequality \eqref{CFineq} can be recovered from \eqref{FSDL} simply by taking $\phi_\alpha \equiv 0$ for all $\alpha$. As in the case of \eqref{CFineq}, for any fixed $1<p<\infty$ and $1<s<2$, Theorem \ref{DLtheorem} does not allow one to recover the full range of Lebesgue space bounds for $T^\Phi$ from those for $M_s$ via \eqref{Bounds}.

One may address this question and obtain optimal control of $T^\Phi$ by combining the ideas developed by Pérez in \cite{Pe94,Pe95} with Di Plinio and Lerner's argument \cite{DiPL}. The main result of this paper is the following.

\begin{theorem}\label{FSmczo}
Let $T^{\Phi}$ be a maximally modulated Calderón--Zygmund operator satisfying \eqref{wt}. Then for any $1<p<\infty$ there is a constant $C<\infty$ such that for any weight $w$
\begin{equation}\label{FSimportant}
\int_{\R^n} |T^{\Phi}f(x)|^p w(x)dx \leq C \int_{\R^n} |f(x)|^p M^{\lfloor p \rfloor +1}w(x)dx.
\end{equation}
\end{theorem}

As may be expected, the constant $C$ in \eqref{FSmA} depends on the exponent $p$ and the assumed weak-type estimate \eqref{wt}; we refer to Section \ref{sec:remarks} for a further discussion.

Of course, one may recover the estimate \eqref{WPineq} from Theorem \ref{FSmczo}. As observed for \eqref{WPineq}, given $1<p<\infty$, the control given by the maximal operator $M^{\lfloor p \rfloor +1}$ is optimal here. Moreover, one cannot replace $\lfloor p \rfloor +1$ by $\lfloor p \rfloor$ in the statement of Theorem \ref{FSmczo} as the resulting inequality is shown to be false for the (unmodulated) Hilbert transform \cite{Pe94}.

Theorem \ref{FSmczo} may be viewed as a corollary of a more precise statement that allows one to replace $M^{\lfloor p \rfloor +1}$ by a sharper class of maximal operators. This strategy originates in Pérez \cite{Pe94} for the case of unmodulated Calderón--Zygmund operators.

Let $A$ be a Young function, that is, $A:[0,\infty) \to [0,\infty)$ is a continuous, convex, increasing function with $A(0)=0$ and such that $A(t)\to \infty$ as $t\to \infty$. We say that a Young function $A$ is doubling if there exists a positive constant $C$ such that $A(2t) \leq C A(t)$ for $t>0$. For each cube $Q \subset \R^n$, we define the Luxemburg norm of $f$ over $Q$ by
$$
\|f\|_{A,Q}=\inf \Big\{\lambda>0: \frac{1}{|Q|}\int_Q A\left(\frac{|f(y)|}{\lambda} \right)dy \leq 1\Big\}
$$
and the maximal operator $M_A$ by
\begin{equation*}
M_Af(x)=\sup_{Q \ni x} \|f\|_{A,Q},
\end{equation*}
where $f$ is a locally integrable function and the supremum is taken over all cubes $Q$ in $\R^n$ containing $x$.\footnote{The Orlicz space $L^A(\R^n)$ consists of all measurable functions with finite finite global Luxemburg norm
$\|f\|_{A}$, where the measure $\frac{\chi_Q}{|Q|}dx$ in the definition of $\|f\|_{A,Q}$ is replaced by $dx$.} In this context we are able to characterize the class of Young functions for which a Fefferman--Stein inequality holds with controlling maximal operator $\mathfrak{M}=M_A$.

\begin{theorem}\label{FSmAthm}
Let $T^{\Phi}$ be a maximally modulated Calderón--Zygmund operator satisfying \eqref{wt}. Suppose that $A$ is a doubling Young function satisfying
\begin{equation}\label{Apcondition}
\int_c^{\infty} \left(\frac{t}{A(t)} \right)^{p'-1}\frac{dt}{t}<\infty
\end{equation}
for some $c>0$. Then for any $1<p<\infty$ there is a constant $C<\infty$ such that for any weight $w$
\begin{equation}\label{FSmA}
\int_{\R^n} |T^{\Phi}f(x)|^p w(x)dx \leq C \int_{\R^n} |f(x)|^p M_Aw(x)dx.
\end{equation}
\end{theorem}
In the unmodulated setting, Pérez \cite{Pe94} pointed out that condition \eqref{Apcondition} is necessary for \eqref{FSmA} to hold for the Riesz transforms. Hence it becomes a necessary condition for Theorem \ref{FSmAthm} to be stated in such a generality, characterizing the class of Young functions for which \eqref{FSmA} holds. It is interesting to observe that the condition \eqref{Apcondition} does not depend on the behaviour of the weak-type norm $\psi(r)$.

Applying our main result to the Carleson operator one may deduce the corresponding Fefferman--Stein weighted inequalities.

\begin{corollary}\label{FSCarleson}
Let $\mathcal{C}$ be the Carleson operator. Then for any $1<p<\infty$ there is a constant $C<\infty$ such that for every weight $w$
\begin{equation}\label{FScarlineq}
\int_{\R} |\mathcal{C}f(x)|^p w(x)dx \leq C \int_{\R} |f(x)|^p M^{\lfloor p \rfloor +1}w(x)dx.
\end{equation}
\end{corollary}

We shall remark that weighted inequalities for the Carleson operator have been previously studied by many authors. Hunt and Young \cite{HY} established the $L^p(w)$ boundedness of $\mathcal{C}$ for $1<p<\infty$ and $w \in A_p$. Later Grafakos, Martell and Soria \cite{GMS} gave new weighted inequalities for weights in $A_\infty$, as well as vector-valued inequalities for $\mathcal{C}$. More recently, Do and Lacey \cite{DL} deduced weighted estimates for a variation norm version of $\mathcal{C}$ in the context of $A_p$ theory that strengthened the results in \cite{HY}. Finally, Di Plinio and Lerner \cite{DiPL} obtained $L^p(w)$ bounds for $\mathcal{C}$ in terms of the $[w]_{A_q}$ constants for $1\leq q \leq p$. Note that inequality \eqref{FScarlineq} does not fall within the scope of the classical $A_p$ theory.

Indeed the oscillatory nature of the Carleson operator and the operators $T^\Phi$ brings to mind weighted inequalities of the form \eqref{GeneralFS} in other oscillatory contexts. Such inequalities have received considerable attention over the last decades, notably with a longstanding conjecture raised by Stein in the context of the disc multiplier. In \cite{Stein79}, Stein suggested that the disc multiplier may be controlled, via a general weighted $L^2$ inequality, by some variant of the universal maximal function
$$
\mathcal{N}w(x):=\sup_{R \ni x} \frac{1}{|R|}\int_{R}w(z)dz,
$$
where the supremum is taken over all rectangles $R$ containing $x$. This conjecture, which is often referred to as \textit{Stein's conjecture}, is far from having a satisfactory answer for $n\geq 2$,\footnote{The case $n=1$ reduces to the study of the Hilbert transform.} although positive results were obtained by Carbery, Romera and Soria \cite{CRS} in the setting of radial weights. Of course Stein's conjecture may be viewed as a departure point towards a generalisation of the weighted inequality \eqref{FScarlineq} to higher dimensions, where we naturally define the $n$-dimensional Carleson operator to be
$$
\mathcal{C}^n f(x):=\sup_{r>0} \left|\int_{|\xi|<r}\widehat{f}(\xi)e^{2\pi ix\cdot \xi}d\xi \right|
$$
for $x \in \R^n$. A similar conjecture to that of Stein's was raised by Córdoba \cite{CordKak} in the context of Bochner--Riesz multipliers; see Carbery and Seeger \cite{CS} or Lee, Rogers and Seeger \cite{LRSw} for results in this direction. For further recent examples of control of highly oscillatory operators by maximal functions we refer to Bennett, Carbery, Soria and Vargas \cite{BCSV}, Bennett and Harrison \cite{BH}, Bennett \cite{Be13}, Córdoba and Rogers \cite{CR2014} and the work of Bennett and the author \cite{BB2015}. The results of this paper may be seen to combine aspects of the more classical inequalities \eqref{FSineq} and \eqref{WPineq} and the oscillatory examples just described.
\\

\noindent \textbf{Structure of the paper.} In Section \ref{sec:dyadic} we present the results of Lerner \cite{LeA2, LeCZ, DiPL} that allow one to bound in norm $T^{\Phi}$ by the so-called dyadic sparse operators. In Section \ref{maximal} we recall the results obtained by Pérez in \cite{Pe94,Pe95} concerning the maximal operator associated to a Young function. In Section \ref{proofthm} we provide the proof of Theorem \ref{FSmAthm} and explain how to apply it to deduce Theorem \ref{FSmczo}. Section \ref{Applications} contains some applications that may be deduced from our main result, and Section \ref{sec:VVextensions} is devoted to a vector-valued extension of the main theorem. In Section \ref{sec:TwoWeighted} we are concerned with an extension of our result to a more general two-weight setting. Finally, we conclude with a section on further remarks.
\\

\noindent \textbf{Acknowledgements.} The author would like to thank his supervisor Jon Bennett for his continuous support and for many valuable comments on the exposition of this paper.

\section{A norm estimate by dyadic sparse operators}\label{sec:dyadic}

Here we present a result in \cite{DiPL} that allows one to reduce the proof of \eqref{FSmA} to a Fefferman--Stein inequality for dyadic sparse operators. This reduction rests on a certain local mean oscillation estimate. Such estimates have been developed by Lerner and other authors and have become a powerful technique over the last few years. See, for instance, \cite{LePointwise,LeCZ,DiPL,Hyt14,HLP}. We have considered it instructive to recall this local mean oscillation estimate approach as it will also be used for the vector-valued extension presented in Section \ref{sec:VVextensions}.

We start by recalling some standard definitions. By a general dyadic grid $\mathcal{D}$ we mean a collection of cubes such that: (i) any $Q \in \mathcal{D}$ has sidelength $2^k$, $k \in \mathbb{Z}$; (ii) for any $Q, R \in \mathcal{D}$, we have $Q \cap R \in \{Q,R,\emptyset\}$; (iii) the cubes of a fixed sidelength $2^k$ form a partition of $\R^n$. Given a cube $Q_0$ we denote by $\CD(Q_0)$ the set of all dyadic cubes with respect to $Q_0$, that is, the cubes obtained by dividing dyadically $Q_0$ and its descendants into $2^n$ subcubes. 

We say that $\mathcal{S}$ is a \textit{sparse} family of cubes if for any cube $Q \in \mathcal{S}$ there is a measurable subset $E(Q)\subseteq Q$ such that $|Q|\leq 2|E(Q)|$ and the sets $\{E(Q)\}_{Q \in \mathcal{S}}$ are pairwise disjoint.

Given a measurable function $f$ and a cube $Q$, the local mean oscillation of $f$ on $Q$ is defined by
$$
\omega_\lambda(f;Q)=\inf_{c \in \R}((f-c)\chi_Q)^*(\lambda|Q|)
$$
for $0<\lambda<1$, where $f^*$ denotes the non-increasing rearrangement of $f$.

The median value of $f$ over a cube $Q$, denoted by $m_f(Q)$, is a nonunique real number such that
$$
|\{x \in Q : f(x)>m_f(Q)\}| \leq |Q|/2
\hspace{0.5cm} \text{and} \hspace{0.5cm}
|\{x \in Q : f(x)<m_f(Q)\}| \leq |Q|/2.
$$

Given a measurable function $f$ and a cube $Q_0$, one may control pointwise the value of $f$ on $Q_0$ in terms of the median value of $f$ on $Q_0$ and the local mean oscillation of $f$ in a sparse family of cubes. A first version of this result was obtained by Lerner \cite{LePointwise}; see \cite{Hyt14} for the following refined version.

\begin{theorem}[\cite{Hyt14}]\label{LernerFormula}
Let $f$ be a measurable function on $\R^n$ and $Q_0$ be a fixed cube. Then there exists a sparse family of cubes $\CS \subset \CD(Q_0)$ such that
$$
|f(x)-m_f(Q_0)| \leq 2 \sum_{Q \in \CS} \omega_{\frac{1}{2^{n+2}}}(f; Q)\chi_Q(x)
$$
for a.e. $x \in Q_0$.
\end{theorem}

Di Plinio and Lerner \cite{DiPL} applied the above local mean oscillation estimate to $T^\Phi f$ to obtain an estimate for $\|T^\Phi f\|_{L^p(w)}$. Then, one is concerned with obtaining a bound for the local mean oscillation of $T^\Phi f$ on a cube $Q$.

\begin{proposition}[\cite{DiPL}]\label{oscillationProp}
Let $T^\Phi$ be a maximally modulated Calderón--Zygmund operator satisfying \eqref{wt}. Then, for any cube $Q \subset \R^d$ and any $1<r\leq r_0$,
\begin{equation}\label{oscillationMMCZO}
\omega_{\lambda}(T^\Phi f; Q) \lesssim \psi(r)\left(\frac{1}{|\bar Q|}\int_{\bar Q} |f|^r \right)^{1/r} + \sum_{m=0}^\infty \frac{1}{2^{m\delta}} \left( \frac{1}{|2^mQ|}\int_{2^mQ}|f| \right),
\end{equation}
where $\bar Q=2\sqrt{n}Q$.
\end{proposition}

Given a sparse family $\mathcal{S}$, the above proposition suggests to consider the dyadic sparse operator
$$
\mathcal{A}_{r,\mathcal{S}}f(x)=\sum_{Q \in \mathcal{S}} \left(\frac{1}{|\bar Q|} \int_{\bar Q} |f|^r \right)^{1/r}\chi_Q(x).
$$
The following norm estimate result allows to deduce boundedness for the operator $T^\Phi$ from uniform boundedness on the dyadic grids $\mathcal{D}$ and the sparse families $\mathcal{S} \subset \mathcal{D}$ for the operator $\mathcal{A}_{r,\mathcal{S}}$.

\begin{proposition}[\cite{DiPL}]\label{L24}
Let $T^{\mathcal{F}}$ be a maximally modulated Calderón--Zygmund operator satisfying \eqref{wt}. Let $1<p<\infty$ and let $w$ be an arbitrary weight. Then
$$
\|T^{\Phi}f\|_{L^p(w)} \lesssim \inf_{1<r\leq r_0} \Big\{ \psi(r)\sup_{\mathcal{D},\mathcal{S}} \|\mathcal{A}_{r,\mathcal{S}}f\|_{L^p(w)} \Big\}
$$
where the supremum is taken over all dyadic grids $\mathcal{D}$ and all sparse families $\mathcal{S}\subset \mathcal{D}$.
\end{proposition}

As we will see in Section \ref{proofthm}, Theorem \ref{FSmAthm} may be deduced from its analogous statement in the context of the dyadic sparse operators $\mathcal{A}_{r,\mathcal{S}}$.

\section{Bounds for the maximal operator}\label{maximal}

Let $B$ be a Young function. We define the complementary Young function $\bar B$ to be the Legendre transform of $B$, that is
$$
\bar B(t)=\sup_{s>0}\{st-B(s)\}, \:\:\:\:\: t>0.
$$
We have that $\bar B$ is also a Young function, and it satisfies
$$
t \leq B^{-1}(t)\bar B^{-1}(t)\leq 2t
$$
for $t>0$. For all functions $f,g$ and all cubes $Q\subset \R^n$, the following generalised Hölder's inequality holds,
$$
\frac{1}{|Q|}\int_{Q}f(x)g(x)dx \leq \|f\|_{B,Q}\|g\|_{\bar B, Q}.
$$
Pérez \cite{Pe95} characterised the Young functions $B$ such that $M_B$ is bounded in $L^p$ for $p>1$; such characterisation is commonly known as $B_p$ condition. He also established that condition to be equivalent to certain weighted inequalities for $M_B$ and related maximal operators.
\begin{theorem}[\cite{Pe95}]\label{PeCharac}
Let $1<p<\infty$. Let $A$ and $B$ be doubling Young functions satisfying $\bar B (t)=A(t^{p'})$. Then the following are equivalent:
\begin{enumerate}[(i)]
\item $B$ satisfies the $B_p$ condition, denoted by $B \in B_p$: there is a constant $c>0$ such that
$$
\int_c^\infty \frac{B(t)}{t^p}\frac{dt}{t} \approx \int_ c^\infty \left(\frac{t^{p'}}{\bar B (t)} \right)^{p-1} \frac{dt}{t}<\infty.
$$

\item There is a constant $c>0$ such that
$$
\int_c^\infty \left(\frac{t}{A(t)} \right)^{p-1}\frac{dt}{t}<\infty.
$$

\item There is a constant $C<\infty$ such that
$$
\int_{\R^n} M_Bf(x)^pdx \leq C \int_{\R^n} f(x)^pdx
$$
for all non-negative functions $f$.

\item There is a constant $C<\infty$ such that
$$
\int_{\R^n}M_Bf(x)^pu(x)dx \leq C \int_{\R^n}f(x)^p Mu(x)dx
$$
for all non-negative functions $f$ and any weight $u$.

\item There is a constant $C<\infty$ such that 
\begin{equation}\label{FSreverse}
\int_{\R^n} Mf(x)^p \frac{u(x)}{(M_Aw(x))^{p-1}}dx \leq C \int_{\R^n} f(x)^p \frac{Mu(x)}{w(x)^{p-1}}dx
\end{equation}
for all non-negative functions $f$ and any weights $u$, $w$.
\end{enumerate}
\end{theorem}
A classical result from Coifman and Rochberg \cite{CR} asserts that for any locally integrable function $w$ such that $Mw(x)<\infty$ a.e. and $0<\delta<1$, the function $(Mw)^{\delta}(x)$ is an $A_1$ weight with constant independent of $w$. More precisely,
$$
M\left((Mw)^{\delta}\right)(x) \leq C_n\frac{1}{1-\delta}(Mw)^{\delta}(x)
$$
for almost all $x \in \R^n$. As Pérez remarks in \cite{Pe94}, the function $(M_Aw)^\delta$ still continues to be an $A_1$ weight for any Young function $A$. Proceeding as Coifman and Rochberg, one may indeed see that the $A_1$ constant is independent of the Young function $A$; we leave the proof to the interested reader.
\begin{proposition}\label{Coifman}
Let $A$ be a Young function. If $0<\delta<1$, then $(M_Aw)^{\delta}$ is an $A_1$ weight with $A_1$ constant independent of $w$. In particular,
$$
M\left((M_Aw)^{\delta}\right)(x) \leq C_{n}\frac{1}{1-\delta}(M_Aw)^{\delta}(x)
$$
for almost all $x \in \R^n$.
\end{proposition}

\section{Proof of Theorems \ref{FSmczo} and \ref{FSmAthm}} \label{proofthm}

In this section we give a proof of Theorem \ref{FSmAthm} and we use it, thanks to an observation due to Pérez \cite{Pe94,Pe95}, to deduce Theorem \ref{FSmczo}. Our proof follows a similar pattern of a proof of Di Plinio and Lerner in \cite{DiPL}.

As seen in Section \ref{sec:dyadic}, the boundedness of $T^{\Phi}$ may be essentially reduced to the uniform boundedness of the dyadic sparse operators $\mathcal{A}_{r,\mathcal{S}}$. In particular, we have the following Fefferman--Stein inequality for $\mathcal{A}_{r,\mathcal{S}}$.

\begin{theorem}\label{SparseBound}
Let $\mathcal{D}$ be a dyadic grid and $\mathcal{S} \subset \mathcal{D}$ a sparse family of cubes. Suppose that $A$ is a Young function satisfying \eqref{Apcondition}.Then for $1<p<\infty$, there is a constant $C_{n,p,A}<\infty$ independent of $\mathcal{S}$, $\mathcal{D}$ and the weight $w$ such that
$$
\|\mathcal{A}_{r,\mathcal{S}}f\|_{L^p(w)} \leq C_{n,p,A}\left(\Big(\frac{p+1}{2r}\Big)' \right)^{1/r} \|f\|_{L^p(M_Aw)}
$$
holds for any $1<r<\frac{p+1}{2}$ and any non-negative function $f$.
\end{theorem}

\begin{proof}
We first linearise the operator $\mathcal{A}_{r,\mathcal{S}}$. For any $Q$, by $L^p$ duality, there exists $g_Q$ supported in $\bar Q$ such that $\frac{1}{|\bar Q|}\int_{\bar Q} g^{r'}_Q=1$ and
$$
\left(\frac{1}{|\bar Q|} \int_{\bar Q}f^r \right)^{1/r}=\frac{1}{|\bar Q|}\int_{\bar Q}fg_Q.
$$
Of course the sequence of functions $\{g_Q\}_Q$ depends on the function $f$. Given such a sequence, we can define a linear operator $L_{f}$ by
$$
L_{f}h(x)=\sum_{Q \in \mathcal{S}} \left(\frac{1}{|\bar Q|}\int_{\bar Q} hg_Q \right)\chi_Q(x).
$$
Note that evaluating in $f$ one recovers $\CA_{r,\mathcal{S}}f$, that is $L_f(f)=\mathcal{A}_{r,\mathcal{S}}f$. Then, in order to obtain an estimate for $\|\mathcal{A}_{r,\mathcal{S}}\|_{L^p(w)}$ independent of $\mathcal{S}$ and $\mathcal{D}$, it is enough to obtain the corresponding estimate for $\|L_f h\|_{L^p(w)}$ uniformly in the functions $g_Q$. For ease of notation we remove the dependence of $f$ in $L_f$. By duality, the estimate
$$
\|Lh\|_{L^p(w)}\leq C_{n,p,A}\left(\Big(\frac{p+1}{2r}\Big)' \right)^{1/r} \|h\|_{L^p(M_Aw)}
$$
is equivalent to
\begin{equation}\label{dual}
\|L^*h\|_{L^{p'}((M_Aw)^{1-p'})} \leq C_{n,p,A}\left(\Big(\frac{p+1}{2r}\Big)' \right)^{1/r} \|h\|_{L^{p'}(w^{1-p'})}
\end{equation}
where $L^*$ denotes the $L^2(\R^n)$-adjoint operator of $L$. Since $A$ satisfies \eqref{Apcondition}, one can apply Theorem \ref{PeCharac} with $p$ replaced by $p'$. Using \eqref{FSreverse} with $u\equiv1$, the estimate \eqref{dual} follows from
\begin{equation}\label{ImportantIneq}
\|L^*h\|_{L^{p'}((M_Aw)^{1-p'})} \leq C_n\left(\Big(\frac{p+1}{2r}\Big)' \right)^{1/r} \|Mh\|_{L^{p'}((M_Aw)^{1-p'})}.
\end{equation}
We focus then on obtaining \eqref{ImportantIneq}. By duality, there exists $\eta \geq 0$ such that $\|\eta\|_{L^p(M_Aw)}=1$ and
$$
\|L^*h\|_{L^{p'}((M_Aw)^{1-p'})} =\int_{\R^n}L^*(h)\eta =\int_{\R^n}hL\eta .
$$
By Hölder's inequality and the $L^{r'}$ boundedness of $g_Q$,
\begin{align}
\int_{\R^n}hL\eta &=\sum_{Q \in \mathcal{S}} \left( \frac{1}{|\bar Q|}\int_{\bar Q} \eta g_Q \right)\int_Q h 
\leq \sum_{Q \in \mathcal{S}} \left( \frac{1}{|\bar Q|}\int_{\bar Q} \eta^r \right)^{1/r} \int_Q h \nonumber
\\
&\leq  \sum_{Q \in \mathcal{S}}  \left( \frac{1}{|\bar Q|}\int_{\bar Q}\eta^r\right)^{1/r}\left( \frac{1}{|\bar Q|}\int_{\bar Q}h\right)(2\sqrt{n})^n |Q| \nonumber
\\
&=(2\sqrt{n})^n \sum_{Q \in \mathcal{S}} \left(\frac{1}{|\bar Q|}\int_{\bar Q} \eta^r \left( \frac{1}{|\bar Q|}\int_{\bar Q}h \right)^{\frac{r}{p+1}} \right)^{1/r} \left( \frac{1}{|\bar Q|}\int_{\bar Q}h\right)^{\frac{p}{p+1}} |Q|. \label{I}
\end{align}
Recall that by definition of the Hardy--Littlewood maximal operator
\begin{equation}\label{HLpointwise}
\frac{1}{|\bar Q|} \int_{\bar Q} h(x)dx \leq Mh(y)
\end{equation}
holds for every $y \in \bar Q$. Combining this and the sparseness of $\mathcal{S}$
\begin{align}
\eqref{I} &\leq 2(2\sqrt{n})^{n} \sum_{Q \in \mathcal{S}} \left(\frac{1}{|\bar Q|} \int_{\bar Q} \left((Mh)^{\frac{1}{p+1}}\eta \right)^r \right)^{1/r} \left( \frac{1}{|\bar Q|}\int_{\bar Q}h\right)^{\frac{p}{p+1}} |E(Q)| \nonumber
\\
&\leq 2(2\sqrt{n})^{n} \sum_{Q \in \mathcal{S}} \int_{E(Q)} M_r((Mh)^{\frac{1}{p+1}}\eta)(Mh)^{\frac{p}{p+1}} \nonumber
\\
&\leq 2(2\sqrt{n})^{n} \int_{\R^n} M_r((Mh)^{\frac{1}{p+1}}\eta)(Mh)^{\frac{p}{p+1}}, \label{II}
\end{align}
where we have used that $(E(Q))_{Q \in \mathcal{S}}$ are pairwise disjoint and that \eqref{HLpointwise} also holds for $y \in E(Q)\subseteq Q \subseteq \bar Q$. By Hölder's inequality with exponents $\rho=\frac{p+1}{2}$ and $\rho'=\frac{p+1}{p-1}$,
\begin{align}
\eqref{II} &=2(2\sqrt{n})^{n}\int_{\R^n} M_r((Mh)^{\frac{1}{p+1}}\eta) (M_Aw)^{\frac{1}{p+1}} (Mh)^{\frac{p}{p+1}} (M_Aw)^{-\frac{1}{p+1}} \nonumber
\\
&\leq 2(2\sqrt{n})^{n}\|M_r((Mh)^{\frac{1}{p+1}}\eta)\|_{L^{\frac{p+1}{2}}((M_Aw)^{1/2})} \|Mh\|^{\frac{p}{p+1}}_{L^{p'}((M_Aw)^{1-p'})}. \label{product}
\end{align}
For $r<\frac{p+1}{2}$, we can apply the classical Fefferman--Stein inequality \eqref{FSineq} to the first term in \eqref{product}
\begin{align*}
\|M_r((Mh)^{\frac{1}{p+1}}\eta)&\|_{L^{\frac{p+1}{2}}((M_Aw)^{1/2})} \\
& \leq C_n\left(\Big(\frac{p+1}{2r}\Big)' \right)^{1/r} \|(Mh)^{\frac{1}{p+1}}\eta\|_{L^{\frac{p+1}{2}}(M((M_Aw)^{1/2}))},
\end{align*}
and by Proposition \ref{Coifman}
$$
\|(Mh)^{\frac{1}{p+1}}\eta\|_{L^{\frac{p+1}{2}}(M((M_Aw)^{1/2}))} \leq C_n\|(Mh)^{\frac{1}{p+1}}\eta\|_{L^{\frac{p+1}{2}}((M_Aw)^{1/2})}.
$$
Finally, by an application of Hölder's inequality with $\rho=2p'$ and $\rho'=\frac{2p}{p+1}$
\begin{align*}
\|(Mh)^{\frac{1}{p+1}}\eta\|_{L^{\frac{p+1}{2}}((M_Aw)^{1/2})}&= \left( \int_{\R^n} \left( (Mh)^{\frac{1}{2}}(M_Aw)^{-\frac{1}{2p}} \right) \left( \eta^{\frac{p+1}{2}} (M_Aw)^{\frac{p+1}{2p}} \right) \right)^{\frac{2}{p+1}}
\\
&\leq \|Mh\|_{L^{p'}((M_Aw)^{1-p'})}^{\frac{1}{p+1}} \|\eta\|_{L^p(M_Aw)} \\
&=\|Mh\|_{L^{p'}((M_Aw)^{1-p'})}^{\frac{1}{p+1}},
\end{align*}
where the last equality holds since $\|\eta\|_{L^p(M_Aw)}=1$. Altogether,
$$
\|L^*h\|_{L^{p'}((M_Aw)^{1-p'})} \leq 2(2\sqrt{n})^n C_n\left(\Big(\frac{p+1}{2r}\Big)' \right)^{1/r} \|Mh\|_{L^{p'}((M_Aw)^{1-p'})}.
$$
This concludes the proof.
\end{proof}

We are now able to prove Theorem \ref{FSmAthm}.

\begin{proof}[Proof of Theorem \ref{FSmAthm}]
By Proposition \ref{L24}, it is enough to show that for any $1<p<\infty$,
\begin{equation*}
\inf_{1<r\leq r_0} \Big\{ \psi(r)\sup_{\mathcal{D},\mathcal{S}} \|\mathcal{A}_{r,\mathcal{S}}f\|_{L^p(w)} \Big\} \lesssim \|f\|_{L^p(M_Aw)}.
\end{equation*}
By Theorem \ref{SparseBound},
\begin{equation*}
\sup_{\mathcal{D}, \mathcal{S}}\|\mathcal{A}_{r,\mathcal{S}}f\|_{L^p(w)} \leq C_{n,p,A}\left(\Big(\frac{p+1}{2r}\Big)' \right)^{1/r} \|f\|_{L^p(M_Aw)}
\end{equation*}
for any $1<r<\frac{p+1}{2}$, since the bound was independent of $\mathcal{D}$, $\mathcal{S}$.

For every $p>1$, consider
$$
r_p=\min\Big\{r_0,1+\frac{p-1}{3}\Big\}=\min \Big\{r_0,\frac{p+2}{3}\Big\}.
$$
We have that $1<r_p\leq r_0$ and $r_p< \frac{p+1}{2}$. Then
$$
\|T^{\Phi}f\|_{L^p(w)} \lesssim \psi(r_p)\sup_{\mathcal{D},\mathcal{S}} \|\mathcal{A}_{r_p,\mathcal{S}}f\|_{L^p(w)} \leq \psi(r_p)C_{n,p,A}\left(\Big(\frac{p+1}{2r_p}\Big)' \right)^{1/r_p} \|f\|_{L^p(M_Aw)}.
$$
\end{proof}

Observe that the proof of Theorem \ref{FSmAthm} may be extended to other operators whose bounds depend in a suitable way on those of $\mathcal{A}_{r,\mathcal{S}}$. This will be the case of the vector-valued extension presented in Section \ref{sec:VVextensions}.

Theorem \ref{FSmczo} may be deduced from Theorem \ref{FSmAthm}. Given the specific Young function $A(t)=t\log^{\lfloor p \rfloor}(1+t)$, which clearly satisfies \eqref{Apcondition}, there exits a constant $C<\infty$ such that $M_A w (x) \leq C M^{\lfloor p \rfloor +1}w(x)$ for any weight $w$. This observation is due to Pérez \cite{Pe94,Pe95}.

\section{Applications}\label{Applications}

\subsection{Maximal multiplier of bounded variation}

The essence of the classical Marcinkiewicz multiplier theorem is the observation that a multiplier of bounded variation on the line often satisfies the same norm inequalities as the Hilbert transform. In particular, if $m$ is a bounded variation multiplier and $T_m$ is its associated operator, one may deduce
$$
\int_{\R} |T_mf(x)|^p w(x)dx \leq C \int_\R |f(x)|^p M^{\lfloor p \rfloor +1}w(x)dx
$$
for any weight $w$. Using Corollary \ref{FSCarleson}, the analogous maximal-multiplier inequality in the sense of Oberlin \cite{Ob} follows. Consider the maximal-multiplier operator
$$
\mathcal{M}_{BV}f(x):=\sup_{m: \|m\|_{BV}\leq 1} |(m\widehat{f})\:\widecheck{}\:(x)|
$$
where the supremum is taken over all functions whose variation norm is less or equal than 1. Recall that the variation norm is defined by
\begin{equation}\label{Vnorm}
\|m\|_{BV}:=\|m\|_{\infty}+ \sup_{N,\xi_0<\cdots<\xi_N} \Big(\sum_{i=1}^N |m(\xi_i)-m(\xi_{i-1})| \Big),
\end{equation}
where the supremum is taken over all strictly increasing finite length sequences of real numbers. The second term in the right hand side of \eqref{Vnorm} is known as the total variation of $m$.
\begin{theorem}
For $1<p<\infty$, there is a constant $C<\infty$ such that
\begin{equation}\label{MaxMult}
\int_\R |\mathcal{M}_{BV}f(x)|^pw(x)dx \leq C \int_\R |f(x)|^p M^{\lfloor p \rfloor +1}w(x)dx
\end{equation}
holds for any weight $w$.
\end{theorem}

\begin{proof}
Since $m$ is of global bounded variation,
\begin{align*}
T_mf(x)=cf(x)+\int_{\R}S_{(t,\infty)}f(x)dm(t) \leq cf(x)+\int_\R \mathcal{C}f(x)dm(t),
\end{align*}
where $(S_{(t,\infty)}f)\: \widehat{\:\: \:}\:(\xi)=\chi_{(t,\infty)}(\xi)\widehat{f}(\xi)$ and $dm$ denotes the Lebesgue--Stieltjes measure associated to $m$.
Then 
$$
\sup_{m: \|m\|_{BV} \leq 1} |T_mf(x)| \leq c |f(x)| + |\mathcal{C}f(x)|\sup_{m: \|m\|_{BV} \leq 1} \int_\R |dm|(t) \leq c|f(x)| + |\mathcal{C}f(x)|,
$$
where the last inequality follows since the integral of $|dm|$ corresponds to the total variation of $m$. The proof concludes by taking $L^p(w)$ norms and using Corollary \ref{FSCarleson}.
\end{proof}

\begin{remark}
Let $m$ be a multiplier of bounded variation and let $m_t(\xi)=m(t\xi)$. Consider the maximal operator associated to these multipliers, that is, 
$$T^*_mf(x)=\sup_{t>0}|(m_t\widehat{f})\:\widecheck{ }\:(x)|.$$
Since $m$ and $m_t$ have the same variation norm, $T^*_mf(x) \leq \|m\|_{BV} \mathcal{M}_{BV}f(x)$, so the inequality \eqref{MaxMult} also holds for $T^*_m$ in place of $\mathcal{M}_{BV}$.
\end{remark}

\subsection{Carleson-type operators in higher dimensions}

Concerning higher dimensional Carleson operators, the Fefferman--Stein weighted inequality also holds for the operator
$$
\mathcal{C}_Pf(x):=\sup_{t>0} \left|\int_{tP}\widehat{f}(\xi)e^{ix\cdot \xi} dx \right|,
$$
where $P$ is a polyhedron with finitely many faces and the origin in its interior. Indeed, Fefferman deduced in \cite{Fe71} that the norm of this operator is bounded by the norm of the one-dimensional Carleson operator $\mathcal{C}$ in any Banach space.

\subsection{The polynomial Carleson operator}

Let $d\in \mathbb{N}$. The polynomial Carleson operator is defined as
\begin{equation}\label{PolyCarleson}
\CC_df(x):=\sup_{\deg(P)\leq d} \left|\pv \int_\R \frac{e^{iP(y)}}{y}f(x-y)dy \right|,
\end{equation}
where the supremum is taken over all real-coefficient polynomials $P$ of degree at most $d$. Note that for $d=1$ one recovers the definition of the Carleson operator.

It was conjectured by Stein that the operator $\CC_d$ is bounded in $L^p$ for $1<p<\infty$. In the case of periodic functions, this conjecture has been recently solved by Lie \cite{Li11} via time-frequency analysis techniques; see \cite{Li09} for his previous work for $\CC_2$.

One may write $\CC_d f(x)= \sup_{P \in \mathcal{P}} |H^\T(\mathcal{M}^{P}f)(x)|$ for $x \in \T$, where $\mathcal{M}^{\mathcal{P}}f(x)=e^{iP(x)}f(x)$ and $H^\T$ denotes the periodic Hilbert transform. Straightforward modifications in the proof of Theorem \ref{FSmczo} yield a similar result for the periodic case and thus, for any $1<p<\infty$ there is a constant $C<\infty$ such that for any weight $w$
\begin{equation*}
\int_{\T} |\mathcal{C}_df(x)|^p w(x)dx \leq C \int_{\T} |f(x)|^p M^{\lfloor p \rfloor +1}w(x)dx.
\end{equation*}

\section{Vector-valued extensions}\label{sec:VVextensions}

Let $T^\Phi$ be a maximally modulated Calderón--Zygmund operator. Given a sequence of functions $f=(f_j)_{j \in \N}$, consider the vector-valued extension of $T^\Phi$, given by $\bar T^\Phi f=(T^\Phi f_j)_{j \in \N}$. For $q \geq 1$, we define the function $|f|_q$ by
$$
|f(x)|_q=\Big(\sum_{j=1}^\infty |f_j(x)|^q\Big)^{1/q}.
$$
As in the case of $T^\Phi$, we will assume that the operator $\bar T^\Phi$ satisfies the a priori weak type inequalities
\begin{equation}\label{vvwt}
\|\bar T^\Phi f\|_{L^{r,\infty}(\ell^q)} \lesssim \psi(r) \|f\|_{L^r(\ell^q)}
\end{equation}
for $1<r\leq r_0$ and some $r_0>1$. Theorem \ref{FSmAthm} naturally extends to $\bar T^\Phi$ in $L^p(\ell^q)$ in the following way.

\begin{theorem}\label{vvThm}
Given $q\geq 1$, let $\bar T^{\Phi}$ be a vector-valued maximally modulated Calderón--Zygmund operator satisfying \eqref{vvwt} and $1<p<\infty$. Suppose that $A$ is a doubling Young function satisfying
\begin{equation*}
\int_c^{\infty} \left(\frac{t}{A(t)} \right)^{p'-1}\frac{dt}{t}<\infty
\end{equation*}
for some $c>0$. Then there is a constant $C<\infty$ such that for any weight $w$
\begin{equation*}
\int_{\R^n} |\bar T^{\Phi}f(x)|_q^p w(x)dx \leq C \int_{\R^n} |f(x)|_q^p M_Aw(x)dx.
\end{equation*}
\end{theorem}

As mentioned in Section \ref{sec:dyadic}, Theorem \ref{vvThm} may be proved via the local mean oscillation estimate approach. In particular, we apply Proposition \ref{LernerFormula} to $|\bar T^\Phi f|_q$. To this end, we need to obtain a bound for the local mean oscillation of $|\bar T^\Phi f|_q$ on a cube $Q$.

\begin{proposition}\label{PropVVoscillation}
Let $q \geq 1$ and $\bar T^\Phi$ be a vector-valued maximally modulated Calderón--Zygmund operator satisfying \eqref{vvwt}. Then, for any $1<r\leq r_0$,
\begin{equation}\label{vvoscillation}
\omega_\lambda(|\bar T^\Phi f|_q; Q) \lesssim \psi(r)\left( \frac{1}{|\bar Q|}\int_{\bar Q} |f|^r_q \right)^{1/r} + \sum_{m=0}^\infty \frac{1}{2^{m\delta}} \left(\frac{1}{|2^m Q|}\int_{2^m Q}|f|_q \right).
\end{equation}
\end{proposition}

Note that the bound obtained in \eqref{vvoscillation} is the same as the one in \eqref{oscillationMMCZO} with $T^\Phi f$ replaced by $|\bar T^\Phi f|_q$ and $f$ replaced by $|f|_q$. One may then obtain an analogue of Proposition \ref{L24} for $|\bar T^\Phi f|_q$ and $|f|_q$, from which Theorem \ref{vvThm} follows after an application of Theorem \ref{SparseBound}.

We proceed now to the proof of Proposition \ref{PropVVoscillation}. The ideas used are quite standard; see for instance \cite{PT} for a similar result in the case of vector-valued Calderón--Zygmund operators.

\begin{proof}[Proof of Proposition \ref{PropVVoscillation}]
Write $f=f^0+f^\infty$, where $f^0=f \chi_{\bar Q}$. Denote by $c_Q$ the centre of the cube $Q$. Then
\begin{align*}
\Big||\bar T^\Phi f(x)|_q-& |\bar T^\Phi  f^\infty(c_Q)|_q\Big| \\ &
 \leq |\bar T^\Phi f (x)-\bar T^\Phi f^\infty (c_Q)|_q \\
& =  \Big(\sum_{j=1}^\infty \Big|\sup_{\alpha \in A}|T(\mathcal{M}^{\phi_\alpha}f_j)(x)|-\sup_{\alpha \in A}|T(\mathcal{M}^{\phi_\alpha}f_j^\infty)(c_Q)|\Big|^q \Big)^{1/q} \\
& \leq \Big(\sum_{j=1}^\infty \sup_{\alpha \in A}|T(\mathcal{M}^{\phi_\alpha}f_j)(x)-T(\mathcal{M}^{\phi_\alpha}f_j^\infty)(c_Q)|^q \Big)^{1/q} \\
& \leq |\bar T^\Phi f^0(x)|_q+ \Big(\sum_{j=1}^\infty \sup_{\alpha \in A}|T(\mathcal{M}^{\phi_\alpha}f_j^\infty)(x)-T(\mathcal{M}^{\phi_\alpha}f_j^\infty)(c_Q)|^q \Big)^{1/q}.
\end{align*}
Since $\bar T^\Phi$ is of weak-type $(r,r)$, it is straightforward to see that
$$
(|\bar T^\Phi f^0|_q \chi_Q)^*(\lambda|Q|) \lesssim \psi(r) \left(\frac{1}{|\bar Q|} \int_{\bar Q} |f|_q^r \right)^{1/r}.
$$
For the second term, if $x \in Q$,
\begin{align*}
\Big(\sum_{j=1}^\infty \sup_{\alpha \in A} & |T(\mathcal{M}^{\phi_\alpha}f_j^\infty)(x)-T(\mathcal{M}^{\phi_\alpha}f_j^\infty)(c_Q)|^q \Big)^{1/q}
\\ &=
\Big(\sum_{j=1}^\infty \sup_{\alpha \in A}\Big|\int_{\R^n \backslash \bar Q}[K(x,z)-K(c_Q,z)]\mathcal{M}^{\phi_\alpha}f_j^\infty(z)dz\Big|^q \Big)^{1/q}
\\
& \leq \Big(\sum_{j=1}^\infty \Big(\int_{\R^n \backslash \bar Q}|K(x,z)-K(c_Q,z)||f_j^\infty(z)|dz\Big)^q \Big)^{1/q} \\
& \leq \int_{\R^n \backslash 2Q} \Big(\sum_{j=1}^\infty |K(x,z)-K(c_Q,z)|^q|f_j^\infty(z)|^q \Big)^{1/q}dz \\
& \leq \int_{\R^n \backslash 2Q} \Big(\sum_{j=1}^\infty \frac{|x-c_Q|^{q\delta}}{|x-z|^{qn+q\delta}}|f_j^\infty(z)|^q \Big)^{1/q}dz \\
& = \sum_{m=1}^\infty \int_{2^{m+1}Q \backslash 2^mQ} \Big(\sum_{j=1}^\infty \frac{|x-c_Q|^{q\delta}}{|x-z|^{qn+q\delta}}|f_j^\infty(z)|^q \Big)^{1/q}dz \\
& \leq \sum_{m=1}^\infty \frac{\ell(Q)^{\delta}}{(2^{m}\ell(Q))^{n+\delta}} \int_{2^{m+1}\backslash 2^m Q} \Big(\sum_{j=1}^\infty|f_j^\infty(z)|^q \Big)^{1/q}dz \\
& \lesssim \sum_{m=0}^\infty \frac{1}{2^{m\delta}} \left(\frac{1}{|2^mQ|}\int_{2^m Q} |f|_q \right),
\end{align*}
where we have used Minkowski integral inequality and the regularity of the kernel $K$. Choosing $c=|\bar T^\Phi f^\infty(c_Q)|_q$ in the definition of $\omega_\lambda(|\bar T^\Phi f|_q; Q)$,
\begin{align*}
\omega_\lambda(|\bar T^\Phi f|_q; Q) & \leq (|\bar T^\Phi f^0|_q\chi_Q)^*(\lambda|Q|) \\
& \:\:\:\:\:\:\:\: + \sup_{x \in Q} \Big| \Big(\sum_{j=1}^\infty \sup_{\alpha \in A}|T(\mathcal{M}^{\phi_\alpha}f_j^\infty)(x)-T(\mathcal{M}^{\phi_\alpha}f_j^\infty)(c_Q)|^q \Big)^{1/q} \Big|\\
& \leq \psi(r) \left(\frac{1}{|\bar Q|} \int_{\bar Q} |f|_q^r \right)^{1/r}+ \sum_{m=0}^\infty \frac{1}{2^{m\delta}} \left(\frac{1}{|2^mQ|}\int_{2^m Q} |f|_q \right).
\end{align*}
\end{proof}

For $q>1$, the vector-valued version of the Carleson operator is bounded on $L^r$ for $r>1$ (see \cite{RRT,GMS}). Thus, for any $1<p,q<\infty$, one may apply Theorem \ref{vvThm} and obtain the weighted inequality
$$
\int_{\R} \Big(\sum_{j=1}^\infty |\CC f_j(x)|^q \Big)^{p/q} w(x) dx \leq C \int_\R \Big( \sum_{j=1}^\infty |f_j(x)|^q \Big)^{p/q} M^{\lfloor p \rfloor +1}w(x)dx,
$$
with $C$ independent of the weight function $w$. The same result follows for the vector-valued version of $\CC_d$, due to its boundedness on $L^r$ for $1<r<\infty$ - see Remarks in \cite{Li11}.

\section{Two-weighted inequalities}\label{sec:TwoWeighted}

Our approach to obtain Fefferman--Stein inequalities may be extended to more general two-weighted inequalities. In the spirit of the work done for the Hardy--Littlewood maximal operator \cite{Pe95} and for Calderón--Zygmund operators \cite{C-UMP2007,LePointwise,LeCZ}, it is possible to deduce sufficient conditions on a pair of weights $(u,v)$ in order to $T^\Phi:L^p(v) \to L^p(u)$. Following the arguments in the proof of Theorem \ref{FSmAthm}, one may prove that if, for some $1<r<\min\{r_0,p\}$, a pair of weights $(u,v)$ satisfy
$$
[u,v]_{A,B}=\sup_{Q \subset \R^n} \|u^{1/p}\|_{A,Q}\|v^{-r/p}\|_{B,Q}^{1/r}<\infty,
$$
where $A$ and $B$ are doubling Young functions such that $\bar A \in B_{p'}$ and $\bar B \in B_{\frac{p+1}{2r}}$, there exists a constant $C=C_{n,p,A,B,u,v}<\infty$ such that
$$
\int_{\R^n} |T^\Phi f(x)|^p u(x)dx \leq C \int_{\R^n}|f(x)|^pv(x)dx.
$$
However, there is an alternative way of obtaining such two-weighted inequalities that does not involve the linearisation and the adjoint operator argument in the proof given for Theorem \ref{FSmAthm}. This approach is along the lines of a two-weighted inequality for Calderón--Zygmund operators proved by Lerner \cite{LeCZ}. Our two-weighted bumped-type inequalities seem to be novel in the setting of the Carleson operator and the related operators discussed in Section \ref{Applications}.

\begin{theorem}\label{2wThm}
Let $T^\Phi$ be a maximally modulated Calderón--Zygmund operator satisfying \eqref{wt}. Let $1<p<\infty$ and $A, B$ be doubling Young functions such that $\bar A \in B_{p'}$ and $\bar B \in B_{p/r}$ for some $1<r<\min\{r_0,p\}$. Let $u$ and $v$ be positive weights such that
$$
\sup_{Q \subset \R^n} \|u^{1/p}\|_{A,Q}\|v^{-r/p}\|^{1/r}_{B,Q} < \infty.
$$
Then there is a constant $C=C_{n,p,A,B,u,v}<\infty$ such that
\begin{equation}\label{2wMMCZOeq}
\int_{\R^n} |T^\Phi f(x)|^p u(x)dx \leq C \int_{\R^n}|f(x)|^pv(x)dx.
\end{equation}
\end{theorem}

\begin{proof}
By Proposition \ref{L24} it is enough to see that
$$
\|\mathcal{A}_{r,\mathcal{S}}f\|_{L^p(u)} \leq C \|f\|_{L^p(v)}
$$
with constant independent of the dyadic sparse family $\CS$. By duality there exists $g \in L^{p'}$, $\|g\|_{p'}=1$ such that
$$
\left(\int_{\R^n} \mathcal{A}_{r,\mathcal{S}}f(x)^p u(x)dx \right)^{1/p}=\int_{\R^n} \mathcal{A}_{r,\mathcal{S}}f(x) u(x)^{1/p}g(x)dx.
$$
Then,
\begin{align*}
\int_{\R^n} (\mathcal{A}_{r,\mathcal{S}}f) u^{1/p}g &= \sum_{Q \in \CS} \left(\frac{1}{|\bar Q|}\int_{\bar Q} |f|^r \right)^{1/r} \int_Q u^{1/p}g \\
& =\sum_{Q \in \CS} \left(\frac{1}{|\bar Q|}\int_{\bar Q} |f|^rv^{r/p} v^{-r/p} \right)^{1/r} \left( \frac{1}{|\bar Q|} \int_{\bar Q} u^{1/p}g \right) |\bar Q| \\
& \lesssim \sum_{Q \in \CS} \|f^r v^{r/p}\|_{\bar B, \bar Q}^{1/r} \|v^{-r/p}\|_{B,\bar Q}^{1/r} \|u^{1/p}\|_{A,\bar Q} \|g\|_{\bar A,\bar Q} |E(Q)| \\
& \leq \sum_{Q \in \CS} \int_{E(Q)} (M_{\bar B}(f^r v^{r/p}))^{1/r}M_{\bar A}g \\
& \leq \int_{\R^n} (M_{\bar B}(f^r v^{r/p}))^{1/r}M_{\bar A}g \\
& \leq \|M_{\bar B}(f^rv^{r/p})\|_{p/r}^{1/r} \|M_{\bar A}g\|_{p'}\\
& \lesssim \|f^rv^{r/p}\|_{p/r}^{1/r} \|g\|_{p'}
\\
& =\|f\|_{L^p(v)},
\end{align*}
where we have used Hölder's inequality for Young functions, the sparseness of the family $\CS$ and the boundedness of the operators $M_{\bar A}$ and $M_{\bar B}$.
\end{proof}

\begin{remark}
The obvious vector-valued extensions considered in Section \ref{sec:VVextensions} also hold for this more general two-weighted case.
\end{remark}

\begin{remark}\label{RecoverFS}
One may recover the Fefferman--Stein weighted inequalities \eqref{FSmA} from Theorem \ref{2wThm} by considering the pair of weights $(w,M_{\Gamma}w)$, where $\Gamma(t)=A(t^{1/p})$ and the Young function $B(t)=t^{(p/r)'+\varepsilon}$, that satisfies $\bar B \in B_{p/r}$. In this case, the constant $C$ in \eqref{2wMMCZOeq} does not depend on $w$, since
\begin{align*}
[w,M_\Gamma w]_{A,B}& =\sup_{Q \subset \R^n} \|w^{1/p}\|_{A,Q}\|(M_\Gamma w)^{-r/p}\|^{1/r}_{B,Q}\\
& = \sup_{Q \subset \R^n} \|w\|_{\Gamma,Q}^{1/p} \left(\frac{1}{|Q|}\int_{Q}(M_\Gamma w)^{-(r/p)((p/r)' + \varepsilon)}\right)^{\frac{1}{(p/r)'+\varepsilon}\frac{1}{r}} \\
& \leq \sup_{Q \subset \R^n} \left(\frac{1}{|Q|}\int_{Q}(M_\Gamma w)^{-(r/p)((p/r)'+\varepsilon)} (M_\Gamma w)^{(1/p)((p/r)'+\varepsilon)r}\right)^{\frac{1}{(p/r)'+\varepsilon}\frac{1}{r}} \\
& = \sup_{Q \subset \R^n} \left( \frac{1}{|Q|} \int_Q 1 \right)^{\frac{1}{(p/r)'+\varepsilon}\frac{1}{r}} \\
& = 1,
\end{align*}
where the second equality follows from the definition of Luxemburg norm.
\end{remark}

An advantage of this alternative proof for inequality \eqref{2wMMCZOeq} is that it can easily be adapted to a multilinear setting, since it does not involve any linear duality. In particular, one may obtain two-weighted inequalities for the multilinear Calderón--Zygmund operators introduced by Grafakos and Torres \cite{GT}, that is, a multilinear operator $T$ bounded from $L^{q_1}\times \cdots \times L^{q_m} \to L^q$ for some $1 \leq q_1, \dots, q_m <\infty$ satisfying $\frac{1}{q}=\frac{1}{q_1}+\cdots \frac{1}{q_m}$ and that can be represented as
$$
T(f_1,\dots,f_m)(x)=\int_{(\R^n)^m}K(x,y_1,\dots,y_m)f_1(y_1)\cdots f_m(y_m)dy_1\cdots dy_m
$$
for all $x\not \in \cap_{j=1}^m \supp f_j$, where the kernel $K:(\R^n)^{m+1} \backslash \Delta \to \R$, with $\Delta=\{(x,y_1,\cdots,y_m) : x=y_1=\cdots = y_m\}$, satisfies the following size condition
$$
|K(y_0,y_1,\dots,y_m)| \leq \frac{A}{(\sum_{k,l=0}^m|y_k-y_l|)^{mn}},
$$
and the regularity condition
$$
|K(y_0,\dots, y_j, \dots, y_m)-K(y_0,\dots, y_j', \dots, y_m)| \leq \frac{A|y_j-y_j'|^{\delta}}{(\sum_{k,l=0}^m|y_k-y_l|)^{mn+\delta}}
$$
for some $\delta>0$ and all $0\leq j \leq m$, whenever $|y_j-y_j'|\leq \frac{1}{2}\max_{0\leq k \leq m}|y_j-y_k|$.

Via a local mean oscillation estimate, Damián, Lerner and Pérez \cite{DLP} reduced norm estimates for such a multilinear operator $T$ to norm estimates for a multilinear version of the dyadic sparse operator $\CA_{r,\CS}$. Then, the proof of Theorem \ref{2wThm} can be adapted to the multilinear Calderón--Zygmund framework; given $1<p,p_1,\dots,p_m<\infty$ satisfying $\frac{1}{p}=\frac{1}{p_1}+\cdots +\frac{1}{p_m}$, and $(u,v_1,\dots,v_m)$ weights such that
$$
\sup_{Q \subset \R^n} \|u^{1/p}\|_{A,Q} \prod_{i=1}^m \|v_i^{-1/p_i}\|_{B_i,Q} <\infty,
$$
where $A,B_1,\dots,B_m$ are doubling Young functions such that $\bar A \in B_{p'}$ and $\bar B_i \in B_{p_i}$, $i=1,\dots,m$, one has
$$
\int_{\R^n}|T(f_1,\dots,f_m)(x)|^pu(x)dx \leq C \prod_{j=1}^m \Big( \int_{\R^n} |f_j(x)|^{p_j}v_j(x)dx \Big)^{p/p_j}.
$$
Of course the multilinear analogue of Remark \ref{RecoverFS} allows one to deduce a Fefferman--Stein weighted inequality for $T$, that is, given $1<p,p_1,\dots,p_m<\infty$ satisfying $\frac{1}{p}=\frac{1}{p_1}+\cdots +\frac{1}{p_m}$ and a doubling Young function $A$ satisfying \eqref{Apcondition},
\begin{equation*}
\int_{\R^n} |T(f_1,\dots,f_m)(x)|^p w(x)dx \leq C \prod_{j=1}^m \Big(\int_{\R^n} |f_j(x)|^{p_j} M_Aw(x)dx \Big)^{p/p_j}.
\end{equation*}
This allows one to recover the result obtained by Hu \cite{Hu2010} via different methods; Hu obtained the above inequality by induction on the level of linearity and using the linear result \eqref{WPineq}.

\section{Further remarks}\label{sec:remarks}

\subsection{Lacunary Carleson operator}Let $\Lambda=\{\lambda_j\}_j$ be a lacunary sequence of integers, that is, $\lambda_{j+1}\geq \theta \lambda_j$ for all $j$ and for some $\theta>1$ and consider the lacunary Carleson maximal operator
$$
\mathcal{C}_{\Lambda}f(x)=\sup_{j \in \mathbb{N}} \left| \pv \int_\R \frac{e^{2\pi i \lambda_j y}}{x-y}f(y)dy \right|.
$$
Of course one has the pointwise estimate $\mathcal{C}_{\Lambda}f(x)\leq \mathcal{C}f(x)$, so the Fefferman--Stein inequality \eqref{FSimportant} trivially holds for $\mathcal{C}_{\Lambda}$. This may be reconciled with a Fefferman--Stein inequality for $\mathcal{C}_\Lambda$ obtained by more classical techniques. Consider the more classical version of the lacunary Carleson operator in terms of the lacunary partial Fourier integrals. Following the lines of \cite{CRV},
$$
S_{\Lambda}^*f(x)=\sup_{k} |S_{\lambda_k}f(x)| \leq cMf(x)+\Big(\sum_{k} |S_{\lambda_k}(f \ast \psi_k)(x)|^2 \Big)^{1/2},
$$
where $\widehat{S_{\lambda_k}f}(\xi):= \chi_{[-\lambda_k, \lambda_k]}(\xi)\widehat{f}(\xi)$, $\psi$ is a suitable Schwartz function, and $\widehat{\psi}_k(\xi):=\widehat{\psi}(\theta^{-k}\xi)$. Since $S_{\lambda_k}$ satisfies the same Lebesgue space inequalities as the Hilbert transform, from the estimate \eqref{WPineq} and weighted Littlewood--Paley theory (which can be obtained via a standard Rademacher function argument and the results from Pérez \cite{Pe94} and Wilson \cite{Wi07}), one may deduce the inequality \eqref{FSimportant} for $\CC_{\Lambda}$ with a higher number of compositions of the Hardy--Littlewood maximal operator $M$.

\subsection{Walsh--Carleson operator}Following the lines of \cite{DiPL}, one may obtain the corresponding Fefferman--Stein inequality for the Walsh--Carleson maximal operator $\mathcal{W}$. This operator is defined as
$$
\mathcal{W}f(x)=\sup_{n \in \mathbb{N}}|\mathcal{W}_nf(x)|,
$$
for $x \in \mathbb{T}=[0,1],$ where $\mathcal{W}_n$ denotes the $n$-th partial Walsh--Fourier sum, often considered as a discrete model of the Fourier case. We refer to \cite{ThWPA} for definitions and elementary results on Walsh--Fourier series. Relying on the weak-type estimate $\|\mathcal{W}f\|_{r,\infty} \lesssim r' \|f\|_{r}$, established in \cite{DiPlinioCollect}, it is proven in \cite{DiPL} that for $1<p<\infty$ and any weight $w$,
$$
\|\mathcal{W}f\|_{L^p(w)} \lesssim \inf_{1<r\leq 2} \Big\{r'\sup_{\mathcal{S}}\|\widetilde{\mathcal{A}}_{r,\mathcal{S}}\|_{L^p(w)}\Big\},
$$
where
$$
\widetilde{\CA}_{r,\mathcal{S}}f(x)=\sum_{Q \in \mathcal{S}} \left(\frac{1}{|Q|}\int_Q |f|^r \right)^{1/r}\chi_Q(x)
$$
and $\mathcal{S}\subset \mathcal{D}(\mathbb{T})$ is a sparse family of dyadic cubes.
Thus,
$$
\int_{\mathbb{T}} |\mathcal{W}f(x)|^p w(x)dx \leq C \int_{\mathbb{T}} |f(x)|^p M^{\lfloor p \rfloor +1}w(x)dx
$$
follows by adapting the proof of Theorem \ref{SparseBound} to the operators $\widetilde{\CA}_{r,\mathcal{S}}$ and to functions defined in $\mathbb{T}$.

\subsection{The $p$-dependence of the constants}
The constant $C$ in our Fefferman--Stein inequalities \eqref{FSmA} depends on the behaviour of the assumed weak-type estimate \eqref{wt}, the Young function $A$, the exponent $p$ and the dimension $n$. In particular,
$$
C=\psi(r_p)C_{n,p,A}\left( \Big(\frac{p+1}{2r_p} \Big)' \right)^{1/r_p},
$$
where $C_{n,p,A}$ is the constant obtained in Theorem \ref{SparseBound} and $r_p=\min\{r_0, \frac{p+2}{3}\}$. A careful inspection of the proof of Theorem \ref{PeCharac} in \cite{Pe95} appears to reveal that, at least for the specific choice of Young function $A(t)=t \log^{\lfloor p \rfloor}(1+t)$, the constant $C_{n,p,A}$ depends exponentially on $p'$. Thus, for the specific weighted inequalities \eqref{FSimportant}, our techniques show that the constant $C$ is bounded by
$$
\frac{C_n}{p-1} \psi(r_{p}) a^{p'}
$$
for certain $a>1$, an object which grows exponentially in $p'$ as $p$ approaches $1$; note that in order for that growth not to be exponential in $p'$ the function $\psi(r)$ should decay exponentially as $r$ approaches $1$, a property that the operators $T^\Phi$ do not satisfy. It would be interesting to determine if a better behaviour on the constant $C$ could be obtained as $p$ approaches $1$, and in particular if the Fefferman--Stein inequalities \eqref{FSimportant} are sensitive to the behaviour of the weak-type norm of $T^\Phi$.

Concerning that weak-type norm, one possible way of obtaining information about the function $\psi$ is by studying the boundedness of the operator $T^\Phi$ near $L^1$. As shown in \cite{DiPL}, if
\begin{equation}\label{nearL1}
\|T^\Phi (f\chi_Q)\|_{L^{1,\infty}(Q)} \lesssim |Q|\|f\|_{\Gamma,Q}
\end{equation}
for each cube $Q \subset \R^n$, where $\Gamma$ is a Young function and the implicit constant is independent of $Q$, then
\begin{equation}\label{inferredBound}
\|T^\Phi f\|_{r,\infty} \lesssim \left( \sup_{t \geq 1} \frac{\Gamma(t)}{t^r} \right)^{1/r} \|f\|_r.
\end{equation}
In the case of the Carleson operator $\mathcal{C}$ the study of the boundedness near $L^1$ is equivalent to determine an Orlicz space $L^\Gamma(\mathbb{T})$, with $L^p(\mathbb{T}) \subsetneq L^\Gamma(\mathbb{T}) \subsetneq L^1(\mathbb{T})$ for all $p>1$, such that there is convergence of Fourier partial sums for functions $f \in L^\Gamma(\mathbb{T})$. It is a well known conjecture that the largest Orlicz space such that convergence holds is $\Gamma(t)=t\log(e+t)$; the current best known result, due to Antonov \cite{Antonov}, ensures that it holds for $\Gamma(t)=t\log(e+t)\log \log \log(e^{e^e}+t)$. An application of Antonov's result in the weak-type estimate \eqref{inferredBound} yields
$$
\|\mathcal{C}f\|_{r,\infty} \lesssim r'\log \log (e^e+r') \|f\|_r,
$$
which seems to be the best weak-type bound for the Carleson operator in the literature. Of course if the \textit{$L\log L$ conjecture} were true one would obtain
\begin{equation*}
\|\mathcal{C}f\|_{r,\infty} \lesssim \frac{1}{r-1}\|f\|_r
\end{equation*}
for $1<r\leq 2$. As discussed above, it would be interesting to determine if the constant $C$ behaves as $(p-1)^{-2}$ as $p$ approaches $1$ in case such conjecture were true.

We should notice that the weak-type bounds provided by the mechanism \eqref{nearL1}--\eqref{inferredBound} are not necessarily sharp. For example, in the case of the lacunary Carleson operator $\mathcal{C}_\Lambda$, in order to establish by such method the weak-type bound
$$
\|\mathcal{C}_\Lambda\|_{r,\infty} \lesssim \log(e+r')\|f\|_r
$$
for $1<r\leq 2$, obtained by Di Plinio in \cite{DiPlinioComptes}, one would need to have a positive answer to a conjecture of Konyagin \cite{Konyagin}, which states that the best Orlicz space such that \eqref{nearL1} holds for $\mathcal{C}_\Lambda$ is given by $\Gamma(t)=t \log \log (e^e+t)$; the current best result is with $\Gamma(t)=t \log\log(e^e+t)\log \log \log \log (e^{e^{e^e}}+t)$, see \cite{DiPlinioCollect, LieLac}.



\bibliographystyle{abbrv} 

\bibliography{BeltranFSCarXiv} 

\end{document}